\documentclass{article}

\usepackage[final]{neurips_2019}

\usepackage[utf8]{inputenc} 
\usepackage[T1]{fontenc}    
\usepackage{hyperref}       
\usepackage{url}            
\usepackage{booktabs}       
\usepackage{amsfonts}       
\usepackage{nicefrac}       
\usepackage{microtype}      
\usepackage{comment}
\usepackage{graphicx}
\usepackage{amsmath}
\usepackage{amsfonts}
\usepackage{amssymb}
\usepackage{amsthm}
\usepackage{color}

\theoremstyle{definition}
\newtheorem{definition}{Definition}
\newtheorem{assumption}{Assumption}

\theoremstyle{plain}
\newtheorem{theorem}{Theorem}

\newtheorem{proposition}{Proposition}
\newtheorem{lemma}{Lemma}
\newtheorem{probstat}{Problem}
\newtheorem*{probstat*}{Problem}

\DeclareMathOperator*{\sign}{sign}

\DeclareMathOperator*{\co}{co}
\DeclareMathOperator*{\cobar}{\overline{co}}

\title{Finite-Time Convergence of Continuous-Time Optimization Algorithms
via Differential Inclusions}

\author{%
  Orlando Romero \\
  Department of Industrial and Systems Engineering\\
  Rensselaer Polytechnic Institute\\
  Troy, NY, USA \\
  \texttt{rodrio2@rpi.edu} \\
   \And
   Mouhacine Benosman \\
   Mitsubishi Electrics Research Laboratories \\
   Cambridge, MA, USA \\
   \texttt{benosman@merl.com} \\
}

\begin{document}

\maketitle

\begin{abstract}
In this paper, we propose two discontinuous dynamical systems in continuous time with guaranteed prescribed finite-time local convergence to strict local minima of a given cost function. Our approach consists of exploiting a Lyapunov-based differential inequality for differential inclusions, which leads to finite-time stability and thus finite-time convergence with a provable bound on the settling time. In particular, for exact solutions to the aforementioned differential inequality, the settling-time bound is also exact, thus achieving prescribed finite-time convergence. We thus construct a class of discontinuous dynamical systems, of second order with respect to the cost function, that serve as continuous-time optimization algorithms with finite-time convergence and prescribed convergence time. Finally, we illustrate our results on the Rosenbrock function.
\end{abstract}

\section{Introduction}
In continuous-time optimization, an ordinary differential equation~(ODE), partial differential equation~(PDE), or differential inclusion is designed in terms of a given cost function, in such a way to lead the solutions to converge (forward in time) to an optimal value of the cost function. To achieve this, tools from Lyapunov stability theory are often employed, mainly because there already exists a rich body of work within the nonlinear systems and control theory community for this purpose. In particular, we seek \emph{asymptotically} Lyapunov stable gradient-based systems with an equilibrium (stationary point) at an isolated extremum of the given cost function, thus certifying local convergence. Naturally, \emph{global} asymptotic stability leads to global convergence, though such an analysis will typically require the cost function to be strongly convex everywhere.

For early work in this direction, see~\citep{botsaris1978,botsaris1978b}, \citep{zghier1981}, \citep{snyman1982,snyman1983}, and \citep{brown1989}. \cite{brockett88} and, subsequently,~\cite{helmke1994}, studied relationships between linear programming, ODEs, and general matrix theory.  Further,~\cite{schropp1995} and \cite{schropp2000} explored several aspects linking nonlinear dynamical systems to gradient-based optimization, including nonlinear constraints. \cite{Cortes2006} proposed two discontinuous normalized modifications of gradient flows to attain \emph{finite-time convergence}. Later,~\cite{Elia2011} proposed a control-theoretic perspective on centralized and distributed convex optimization. 

More recently, \cite{Su2014} derived a second-order ODE as the limit of Nesterov's accelerated gradient method, when the gradient step sizes vanish. This ODE is then used to study Nesterov's scheme from a new perspective, particularly in an larger effort to better understand acceleration without substantially increasing computational burden. 
Expanding upon the aforementioned idea, \cite{franca2018} derived a second-order ODE that models the continuous-time limit of the sequence of iterates generated by the alternating direction method of multipliers (ADMM).Then, the authors employ Lyapunov theory to analyze the stability at critical points of the dynamical systems and to obtain associated convergence rates. 

Later, \cite{Franca2019} analyze general non-smooth and linearly constrained optimization problems by deriving equivalent (at the limit) non-smooth dynamical systems related to variants of the relaxed and accelerated ADMM. Then, the authors employ Lyapunov theory to analyze the stability at critical points of the dynamical systems and to obtain associated convergence rates. Later, \cite{Franca2019} analyze general non-smooth and linearly constrained optimization problems by deriving equivalent (at the limit) non-smooth dynamical systems related to variants of the relaxed and accelerated ADMM. 

In the more traditional context of machine learning, only a few papers have adopted the approach of explicitly borrowing or connecting ideas from control and dynamical systems. For unsupervised learning,~\cite{PLUMBLEY199511} proposes Lyapunov stability theory as an approach to establish convergence of principal component algorithms. \cite{Pequito2011UnsupervisedLO} and~\cite{Aquilanti2019} propose continuous-time generalized expectation-maximization (EM) algorithms, based on mean-field games, for clustering of finite mixture models. \cite{romero2019} establish convergence of the EM algorithm, and a class of generalized EM algorithms denoted $\delta$-EM, via discrete-time Lyapunov stability theory. For supervised learning,~\cite{Liu2019} provide a review of deep learning from the perspective of control and dynamical systems, with a focus in optimal control. \cite{Zhu2018AnOC} and~\cite{Rahnama2019ConnectingLC} explore connections between control theory and adversarial machine learning.

\subsection*{Statement of Contribution}
In this work, we provide a Lyapunov-based tool to both check and construct continuous-time dynamical systems that are finite-time stable and thus lead to finite-time convergence of the candidate Lyapunov function (intended as a surrogate to a given cost function) to its minimum value. In particular, we first extend one of the existing Lyapunov-based inequality condition for finite-time convergence of the usual Lipschitz continuous dynamical systems, to the case of arbitrary differential inclusions. We then use this condition to construct a family of {\it discontinuous, second-order flows}, which guarantee local convergence to a local minimum, in {\it prescribed finite time}. One of the proposed families of continuous-time optimization algorithms is tested on a well-known optimization testcase, namely, the {\it Rosenbrock function}.

\section{Finite-Time Convergence in Optimization via Finite-Time Stability}
Consider some objective cost function $f:\mathbb{R}^n\to\mathbb{R}$ that we wish to minimize. In particular, let $x^\star\in\mathbb{R}^n$ be an arbitrary local minimum of $f$ that is unknown to us.  In continuous-time optimization, we typically proceed by designing a nonlinear state-space dynamical system
\begin{equation}
    \dot{x} = F(x),
    \label{eq:genericsystem1st1}
\end{equation}
or a time-varying one replacing $F(x)$ with $F(t,x)$, for which $F(x)$ can be computed without explicit knowledge of $x^\star$ and for which~\eqref{eq:genericsystem1st1} is certifiably asymptotically stable at $x^\star$. Ideally, computing $F(x)$ should be possible using only up to second-order information on $f$.

In this work, however, we seek dynamical systems for which~\eqref{eq:genericsystem1st1} is certifiably \emph{finite-time} stable at $x^\star$. As will be clear later, such systems need to be possibly discontinuous or non-Lipschitz, based on differential inclusions instead of ODEs. Our approach to achieve this objective is largely based on exploiting the Lyapunov-like differential inequality
\begin{equation}
    \dot{\mathcal{E}}(t) \leq -c\,\mathcal{E}(t)^\alpha,\quad \textnormal{a.e. } t\geq 0,
    \label{eq:ineqEfirst1}
\end{equation}
with constants $c>0$ and $\alpha < 1$, for absolutely continuous functions $\mathcal{E}$ such that $\mathcal{E}(0)>0$. Indeed, under the aforementioned conditions, $\mathcal{E}(t)\to 0$ will be reached in finite time $t\to t^\star\leq \frac{\mathcal{E}(0)^{1-\alpha}}{c(1-\alpha)}<\infty$. 

We now summarize the problem statement:
\begin{probstat}
\label{probstat1}
Given a sufficiently smooth cost function $f:\mathbb{R}^n\to\mathbb{R}$ with a sufficiently regular local minimizer $x^\star$, solve the following tasks:
\begin{enumerate}
    \item Design a sufficiently smooth\footnote{At least locally Lipschitz continuous and \emph{regular} (see Definition~\ref{def:regularfunction} of the supplementary material, Appendix~\ref{subsec:filippov}).} candidate Lyapunov function  $V$ for which is defined and positive definite near and w.r.t. $x^\star$\footnote{In other words, there exists some open neighborhood $\mathcal{D}$ of $x^\star$ such that $V$ is defined in $\mathcal{D}$ and satisfies $V(x^\star) = 0$ and $V(x)>0$ for every $x\in\mathcal{D}\setminus\{x^\star\}$.}.
\item Design a (possibly discontinuous) system\footnote{Right-hand side defined at least a.e., Lebesgue measurable, and locally essentially bounded.}~\eqref{eq:genericsystem1st1} such that $F(x)$ can be computed near $x=x^\star$ without explicit knowledge of $x^\star$, and, given any Filippov solution\footnote{See supplementary material, Appendix~\ref{subsec:filippov}.} $x(\cdot)$ of~\eqref{eq:genericsystem1st1} with \mbox{$x(0)=x_0\neq x^\star$}, the differential inequality~\eqref{eq:ineqEfirst1} is satisfied for $\mathcal{E}(t)\triangleq V(x(t))$.
\end{enumerate}
\end{probstat}

By following this strategy, we will therefore achieve (local and strong) \emph{finite-time stability}, and thus finite-time convergence. Furthermore, if $V(x_0)$ can be upper bounded, then $F$ can be readily tuned to achieve finite-time convergence under a \emph{prescribed} range for the settling time, or even with \emph{exact} prescribed settling time if $V(x_0)$ can be explicitly computed and~\eqref{eq:ineqEfirst1} holds exactly.

\section{A Family of Finite-Time Stable, Second-Order Optimization Flows}
We now propose a family of second-order optimization methods with finite-time convergence constructed using two \emph{gradient-based} Lyapunov functions, namely $V(x) = \|\nabla f(x)\|^2$ and $V(x) = \|\nabla f(x)\|_1$. First, we need to assume sufficient smoothness on the cost function.
\begin{assumption}
$f:\mathbb{R}^n\to\mathbb{R}$ is twice continuously differentiable and strongly convex in an open neighborhood $\mathcal{D}\subseteq\mathbb{R}^n$ of a stationary point $x^\star\in\mathbb{R}^n$.
\label{ass:C2Hessian1}
\end{assumption}

Since $\nabla V(x) = 2\nabla^2 f(x)\nabla f(x)$ for $V(x) = \|\nabla f(x)\|^2$ and $\nabla V(x) = \nabla^2 f(x)\sign(\nabla f(x))$ a.e. for $V(x) = \|\nabla f(x)\|_1$, we can readily design Filippov differential inclusions that are finite-time stable at $x^\star$. In particular, we may design such differential inclusions to achieve an \emph{exact} and \emph{prescribed} finite settling time, at the trade-off of requiring second-order information on $f$.

Given a symmetric and positive definite matrix $P\in\mathbb{R}^{n\times n}$ with SVD decomposition \mbox{$P = V\Sigma V^\top$}, $\Sigma = \textnormal{diag}(\lambda_1\ldots,\lambda_n)$, $\lambda_1,\ldots,\lambda_n>0$, we define \mbox{$P^r \triangleq V\Sigma^r V^\top$}, where $\Sigma^r \triangleq \textnormal{diag}\,(\lambda_1^r,\ldots,\lambda_n^r)$.

\begin{theorem}
Let $c>0$, $p\in [1,2)$, and $r\in\mathbb{R}$. Under Assumption~\ref{ass:C2Hessian1}, any maximal Filippov solution to the discontinuous second-order generalized Newton-like flows
\begin{equation}
    \dot{x} = -c\|\nabla f(x)\|^p
    \frac{[\nabla^2 f(x)]^r\nabla f(x)}{\nabla f(x)^\top[\nabla^2 f(x)]^{r+1}\nabla f(x)}
    \label{eq:GNF11}
\end{equation}
and
\begin{equation}
    \dot{x} = -c\|\nabla f(x)\|_1^{p-1}
    \frac{[\nabla^2 f(x)]^r\sign(\nabla f(x))}{\sign(\nabla f(x))^\top[\nabla^2 f(x)]^{r+1}\sign(\nabla f(x))}
    \label{eq:GNF22}
\end{equation}
with sufficiently small $\|x_0-x^\star\|>0$ (where $x_0=x(0)$) 
will converge in finite time to $x^\star$. Furthermore, their convergence times are given exactly by
\begin{equation}
    t^\star = \frac{\|\nabla f(x_0)\|^{2-p}}{c(2-p)}\qquad\quad t^\star = \frac{\|\nabla f(x_0)\|_1^p}{c\,p} 
    \label{eq:GNF12}
\end{equation}
for \eqref{eq:GNF11}-\eqref{eq:GNF22}, respectively, 
where $x_0 = x(0)$. In particular, given any compact and positively invariant subset $S\subset\mathcal{D}$, both flows converge in finite with the aforementioned settling time upper bounds (which can be tightened by replacing $\overline{\mathcal{D}}$ with $S$) for any $x_0\in S$. Furthermore, if $\mathcal{D}=\mathbb{R}^n$, then we have global finite-time convergnece, \emph{i.e.} finite-time convergence to any maximal Filippov solution~$x(\cdot)$ with arbitrary $x_0=x(0)\in\mathbb{R}^n$.
\end{theorem}

\begin{proof}
See supplementary material, appendix~\ref{subsec:proofthm1}.
\end{proof}

\section{Numerical Experiment: Rosenbrock Function}
We will now test one of our proposed flows on the Rosenbrock function $f:\mathbb{R}^2\to\mathbb{R}$, given by
\begin{equation}
f(x_1,x_2) = (a-x_1)^2 + b(x_2-x_1^2)^2,
\end{equation}
with parameters $a,b\in\mathbb{R}$. This function is nonlinear and non-convex, but smooth. It possesses exactly one stationary point $(x_1^\star,x_2^\star) = (a,a^2)$ for $b\geq 0$, which is a strict global minimum for $b > 0$. If $b<0$, then $(x_1^\star,x_2^\star)$ is a saddle point. Finally, if  $b=0$, then $\{(a,x_2): x_2\in\mathbb{R}\}$ are the stationary points of $f$, and they are all non-strict global minima.

\begin{figure}[!ht]
    \centering
    \hspace{-1cm}
    \includegraphics[width=10cm,height=6cm] {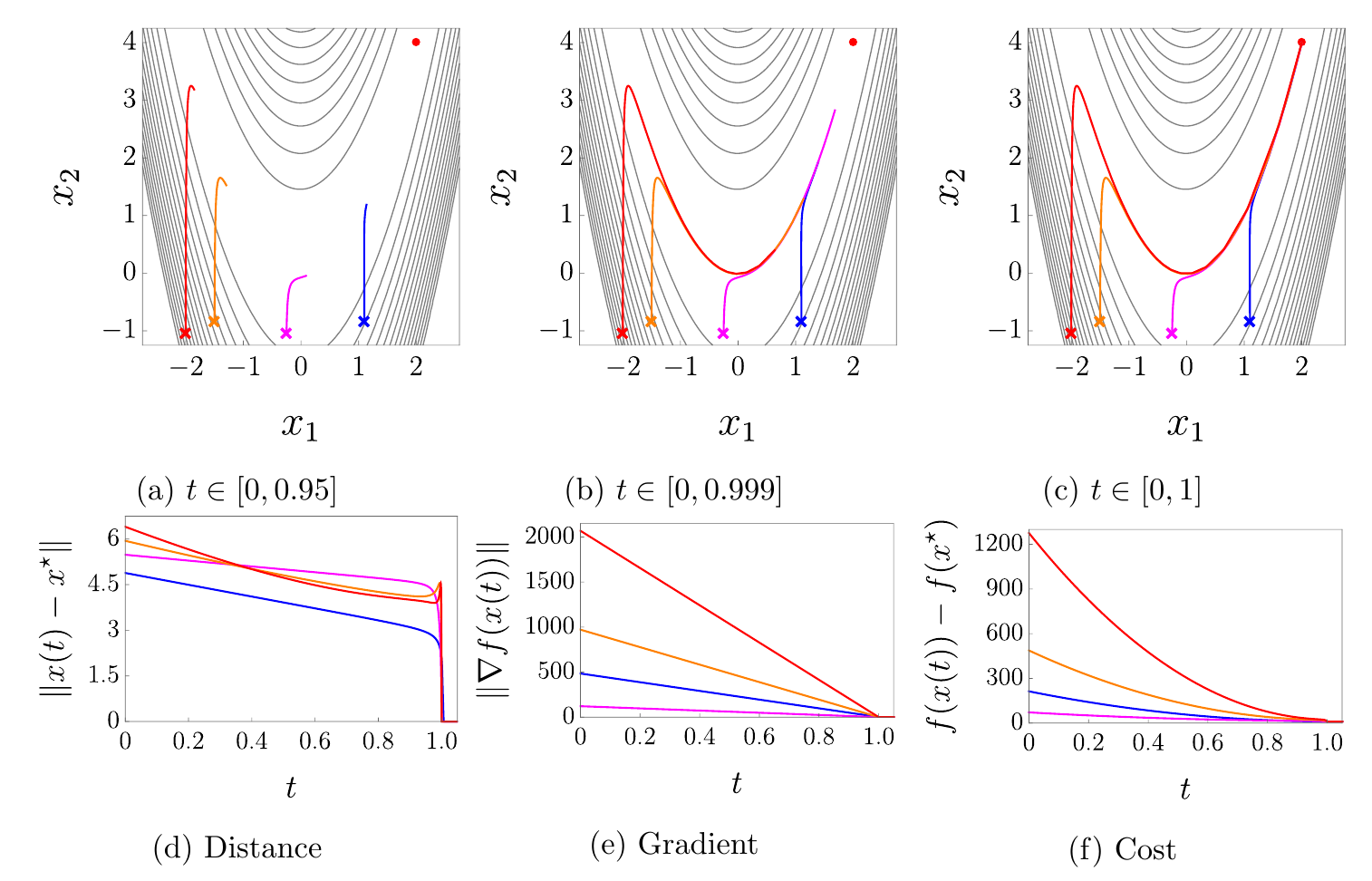}
    \caption{Trajectories of the proposed flow~\eqref{eq:GNF11} with $(c,p,r)=(\|\nabla f(x_0)\|,1,-1)$ and four different initial conditions $x_0\in\mathbb{R}^2$, for the Rosenbrock function with parameters $(a,b)=(2,50)$, which has a unique minimum $x^\star = (a,a^2) = (2,4)$.}
    \label{fig:test1}
\end{figure}

As we can see in Figure~\ref{fig:test1}, this flow converges correctly to the minimum $(a,a^2) = (2,4)$ from all the tested initial conditions with an exact prescribed settling time $t^\star = 1$. It should be noted that at any given point in the trajectory $x(\cdot)$, the functions $t\mapsto \|x(t)-x^\star\|$ and $t\mapsto |f(x(t)) - f(x^\star)|$ are not guaranteed to decrease or remain constant, indeed only $t\mapsto \|\nabla f(x(t))\|$ can be guaranteed to do so, which explains the increase in Figure~\ref{fig:test1}-(d) that could never have occurred in Figure~\ref{fig:test1}-(e).

\section{Conclusions and Future Work}\label{sec:conclusion}
We have introduced a new family of second-order flows for continuous-time optimization. The main characteristic of the proposed flows is their \emph{finite-time} convergence guarantees. Furthermore, they are designed in such a way that the finite convergence time can be prescribed by the user. To be able to analyze these discontinuous flows, we resorted to establishing finite-time stability. In order to do this, we first extended an existing sufficient condition  for finite-time stability through a \mbox{Lyapunov-based} inequality, in the case of smooth dynamics, to the case of non-smooth dynamics modeled by differential inclusions. One of the proposed families was tested on a well-known optimization benchmark -- the Rosenbrock function.

While the obtained results are encouraging, it should be clear that currently available numerical solvers for our proposed flows do not translate into competitive \emph{iterative} optimization algorithms. Deciding how to best discretize a given continuous-time optimization algorithm (i.e. as given by an ODE, PDE, or differential inclusion), or how to compare two such continuous-time algorithms in terms of the performance of corresponding iterative schemes suitable for digital computers, largely remains an open problem, but nonetheless a very active topic of research as of right now.

For the aforementioned reasons, future work will be dedicated to studying discretization of our proposed flows (and other finite-time stable flows) that will hopefully lead to either accelerated schemes when compared to currently available methods, or to a better understanding of the intrinsic boundaries of acceleration strategies. Furthermore, we will extend our results to constrained optimization; gradient-free optimization; time-varying optimization; and extremum-seeking control where the derivative of the cost function is estimated from direct measurements of the cost.

\newpage

\subsubsection*{Acknowledgments}
The majority of the research that led to this work was conducted when the first author was doing an internship at Mitsubishi Electrics Research Laboratories (MERL) in the summer of 2019, under the supervision of the second author.

\bibliography{main}
\bibliographystyle{aaai}
\newpage

\appendix
\section*{Supplementary Material}
\section{Discontinuous Systems and Filippov Differential Inclusions}
\label{subsec:filippov}
Recall that for an initial value problem (IVP)
\begin{subequations}
\begin{align}
\dot{x}(t) &= F(x(t)) \label{eq:genericsystem} \\
x(0) &= x_0
\end{align}
\label{eq:IVP}
\end{subequations}
with $F:\mathbb{R}^n\to\mathbb{R}^n$, the typical way to check for existence of solutions is by establishing continuity of $F$. Likewise, to establish unicity of solution, we typically seek Lipschitz continuity. When $F$ is discontinuous, but nonetheless Lebesgue measurable and locally essentially uniformly bounded, we may understand~\eqref{eq:genericsystem} as the Filippov differential inclusion
\begin{equation}
\dot{x}(t) \in \mathcal{K}[F](x(t)),
\label{eq:filippov}
\end{equation}
where $\mathcal{K}[F]:\mathbb{R}^n\rightrightarrows\mathbb{R}^n$ denotes the Filippov set-valued map given by
\begin{equation}
\mathcal{K}[F](x) \triangleq \bigcap_{\delta > 0}\bigcap_{\mu(S)=0}\cobar F(B_\delta(x)\setminus S),
\label{eq:filippovsetvaluedmap}
\end{equation}
where $\mu$ denotes the usual Lebesgue measure and $\cobar$ the convex closure, \emph{i.e.} closure of the convex hull $\co$. For more details, se~\citep{Paden1987}, from which this section is based on.
\begin{assumption}
$F:\mathbb{R}^n\to\mathbb{R}^n$ is defined a.e. and is Lebesgue measurable in a \mbox{non-empty} open region $U\subset \mathbb{R}^n$. Further, $F$ is locally essentially bounded, \emph{i.e.}, for every point $x\in\mathbb{R}^n$, $F$ is bounded a.e. in some bounded neighborhood of $x$.
\label{ass:existenceFilippov}
\end{assumption}

\begin{definition}[Filippov]
Under Assumption~\ref{ass:existenceFilippov}, we say that $x:[0,\tau)\to\mathbb{R}^n$ with $0<\tau\leq\infty$ is a \emph{Filippov solution} to \eqref{eq:IVP} if $x(\cdot)$ is absolutely continuous, $x(0) = x_0$, and \eqref{eq:filippov} holds a.e. in every compact subset of $[0,\tau)$. Furthermore, we say that $x:[0,\tau)\to\mathbb{R}^n$ is a \emph{maximal} Filippov solution if no other Filippov solution $x'(\cdot)$ exists with $x = x'|_{[0,\tau)}$.
\end{definition}

For a comprehensive overview of discontinuous systems, including sufficient conditions for existence (Proposition 3) and uniqueness (Propositions 4 and 5) of Filippov solutions, see~\citep{Cortes2008}. In particular, it can be established that Filippov solutions to~\eqref{eq:IVP} exist, provided that Assumption~\ref{ass:existenceFilippov} holds.


\begin{proposition}[\cite{Paden1987}, Theorem 1]
Under Assumption~\ref{ass:existenceFilippov},~\eqref{eq:filippovsetvaluedmap} can be computed as
\begin{equation}
\mathcal{K}[F](x) = \left\{\lim_{k\to\infty} F(x_k): x_k\not\in\mathcal{N}_F\cup S, x_k\to x\right\}
\end{equation}
for some (Lebesgue) zero-measure set $\mathcal{N}_F\subset\mathbb{R}^n$ and any other zero-measure set $S\subset\mathbb{R}^n$. In particular, if $F$ is continuous at a fixed $x$, then \mbox{$\mathcal{K}[F](x) = \{F(x)\}$}.
\end{proposition}

For instance, for a gradient flow, we have \mbox{$\mathcal{K}[-\nabla f](x) =\{-\nabla f(x)\}$} for every $x\in\mathbb{R}^n$, provided that $f$ is continuously differentiable. Furthermore, if $f$ is only locally  Lipschitz continuous and regular (see Definition~\ref{def:regularfunction} of Appendix~\ref{subsec:FTSsemiS}), then $\mathcal{K}[-\nabla f](x) = -\partial f(x)$, where
\begin{equation}
    \partial f(x) \triangleq \left\{\lim_{k\to\infty} \nabla f(x_k): x_k\not\in S\cup\mathcal{N}_f, x_k\to x\right\}
\end{equation}
denotes Clarke's generalized gradient~\citep{Clarke1981} of $f$, with $S$ being any zero-measure set and $\mathcal{N}_f$ the zero-measure set over which $f$ is not differentiable.

\section{Finite-Time Stability of Differential Inclusions}
\label{subsec:FTSsemiS}

Consider a general differential inclusion (see \citep{Bacciotti1999} for more details)
\begin{equation}
\dot{x}(t) \in K(x(t))
\label{eq:differentialinclusion}%
\end{equation}
where $K:\mathbb{R}^n\rightrightarrows\mathbb{R}^n$ is a set-valued map.

\begin{assumption}
$K:\mathbb{R}^n\rightrightarrows\mathbb{R}^n$ has nonempty, compact, and convex values, and is \mbox{\emph{upper semi-continuous}}.
\label{ass:K}
\end{assumption}

\cite{Filippov1988} proved that, under Assumption~\ref{ass:existenceFilippov}, the Filippov set-valued map $K = \mathcal{K}[F]$ satisfies the conditions of Assumption~\ref{ass:K}.

Similarly to the previous case of Filippov solutions, we say that $x: [0,\tau)\to\mathbb{R}^n$ with $0<\tau\leq\infty$ is a \mbox{\emph{Carath\'{e}odory solution}} to~\eqref{eq:differentialinclusion} if $x(\cdot)$ is absolutely continuous and~\eqref{eq:differentialinclusion} is satisfied a.e. in every compact subset of $[0,\tau)$. Furthermore, we say that $x(\cdot)$ is a \emph{maximal} Carath\'{e}odory solution if no other Carath\'{e}odory solution $x'(\cdot)$ exists with $x = x'|_{[0,\tau)}$.

We say that $x^\star\in\mathbb{R}^n$ is an \emph{equilibrium} of~\eqref{eq:differentialinclusion} if $x(t) \equiv x^\star$ on some small enough non-degenerate interval is a Carath\'{e}odory solution to~\eqref{eq:differentialinclusion}. In other words, if and only if $0\in F(x^\star)$. 
We say that~\eqref{eq:differentialinclusion} is \emph{(Lyapunov) stable} at $x^\star\in\mathbb{R}^n$ if, for every $\varepsilon > 0$, there exists some $\delta > 0$ such that, for every maximal Carath\'{e}odory solution $x(\cdot)$ of~\eqref{eq:differentialinclusion}, we have $\|x_0 - x^\star\| < \delta \implies \|x(t) - x^\star\| < \varepsilon$ for every $t\geq 0$ in the interval where $x(\cdot)$ is defined. Note that, under Assumption~\ref{ass:K}, if~\eqref{eq:differentialinclusion} is stable at $x^\star$, then $x^\star$ is an equilibrium of~\eqref{eq:differentialinclusion}~\citep{Bacciotti1999}. Furthermore, we say that~\eqref{eq:differentialinclusion} is \emph{(locally and strongly) asymptotically stable} at $x^\star\in\mathbb{R}^n$ if is stable at $x^\star$ and there exists some $\delta > 0$ such that, for every maximal Carath\'{e}odory solution $x:[0,\tau)\to\mathbb{R}^n$ of~\eqref{eq:differentialinclusion}, if $\|x_0 - x^\star\| < \delta$ then $x(t)\to x^\star$ as $t\to\tau$. Finally,~\eqref{eq:differentialinclusion} is \emph{(locally and strongly) finite-time stable} at $x^\star$ if it is asymptotically stable and there exists some $\delta>0$ and $T:B_\delta(x^\star)\to [0,\infty)$ such that, for every maximal Carath\'{e}odory solution $x(\cdot)$ of~\eqref{eq:differentialinclusion} with $x_0\in B_\delta(x^\star)$, we have $\lim_{t\to T(x_0)}x(t) = x^\star$.

We will now construct a Lyapunov-based criterion adapted from the literature of finite-time stability of Lipschitz continuous systems. To do this, we first adapt Lemma~1 in~\citep{Benosman2009} for absolutely continuous functions and non-positive exponents.

\begin{lemma}
Let $\mathcal{E}(\cdot)$ be an absolutely continuous function satisfying the differential inequality
\begin{equation}
    \dot{\mathcal{E}}(t) \leq -c\,\mathcal{E}(t)^\alpha
    \label{eq:Edotineqlemma}
\end{equation}
a.e. in $t\geq 0$, with $c,\mathcal{E}(0)>0$ and $\alpha < 1$. Then, there exists some $t^\star > 0$ such that $\mathcal{E}(t)>0$ for $t\in [0,t^\star)$ and $\mathcal{E}(t^\star) = 0$. Furthermore, $t^\star > 0$ can be bounded by
\begin{equation}
    t^\star \leq \frac{\mathcal{E}(0)^{1-\alpha}}{c(1-\alpha)},
    \label{eq:tstarboundlemma}
\end{equation}
and this bound is exact provided~\eqref{eq:Edotineqlemma} is exact as well. Finally, if~\eqref{eq:Edotineqlemma} is exact and $\alpha\geq 1$, then $t^\star = \infty$ with $\mathcal{E}(\infty) \triangleq \lim_{t\to\infty}\mathcal{E}(t)$.
\label{lemma:Edotineq}
\end{lemma}

\begin{proof}
Suppose that $\mathcal{E}(t)>0$ for every $t\in [0,T]$ with $T > 0$. Let $t^\star$ be the supremum of all such $T$'s, thus satisfying $\mathcal{E}(t) > 0$ for every $t\in [0,t^\star)$. We will now investigate $\mathcal{E}(t^\star)$. First, by continuity of $\mathcal{E}$, it follows that $\mathcal{E}(t^\star) \geq 0$. Now, by rewriting
\begin{equation}
    \dot{\mathcal{E}}(t) \leq -c\,\mathcal{E}(t)^\alpha \iff \frac{d}{dt}\left[\frac{\mathcal{E}(t)^{1-\alpha}}{1-\alpha}\right] \leq -c,
\end{equation}
a.e. in $t\in [0,t^\star)$, we can thus integrate to obtain
\begin{equation}
    \frac{\mathcal{E}(t)^{1-\alpha}}{1-\alpha} - \frac{\mathcal{E}(0)^{1-\alpha}}{1-\alpha} \leq -c\,t,
\end{equation}
everywhere in $t\in [0,t^\star)$, which in turn turn leads to
\begin{equation}
    \mathcal{E}(t) \leq [\mathcal{E}(0)^{1-\alpha} - c(1-\alpha)t]^{1-\alpha}
    \label{eq:Vdotinesol}
\end{equation}
and
\begin{equation}
    t \leq \frac{\mathcal{E}(0)^{1-\alpha} - \mathcal{E}(t)^{1-\alpha}}{c(1-\alpha)} \leq \frac{\mathcal{E}(0)^{1-\alpha}}{c(1-\alpha)},
    \label{eq:tboundsol}
\end{equation}
where the last inequality follows from $\mathcal{E}(t)>0$ for every $t\in [0,t^\star)$. Taking the supremum in~\eqref{eq:tboundsol} then leads to the upper bound~\eqref{eq:tstarboundlemma}. Finally, we conclude that $\mathcal{E}(t^\star) = 0$, since $\mathcal{E}(t^\star) > 0$ is impossible given that it would mean, due to continuity of $\mathcal{E}$, that there exists some $T>t^\star$ such that $\mathcal{E}(t)>0$ for every $t\in [0,T]$, thus contradicting the construction of $t^\star$. 

Finally, notice that if $\mathcal{E}$ satisfies~\eqref{eq:Edotineqlemma} exactly, then~\eqref{eq:Vdotinesol} and the first inequality in~\eqref{eq:tboundsol} are both exact as well. The exactness of the bound~\eqref{eq:tstarboundlemma} thus follows immediately. Furthermore, notice that if~$\alpha\geq 1$, and $\mathcal{E}$ is an exact solution to~\eqref{eq:Edotineqlemma}, \emph{i.e.} $\mathcal{E}(t) = [\mathcal{E}(0)^{1-\alpha} - c(1-\alpha)t]^{1-\alpha}$, then clearly $\mathcal{E}(t)>0$ for every $t\geq 0$ and $\mathcal{E}(t)\to 0$ as $t\to \infty$.
\end{proof}

In their Proposition 2.8, \cite{Cortes2005} proposed a Lyapunov-based criterion to establish finite-time stability of discontinuous systems, which fundamentally boils down to Lemma~\ref{lemma:Edotineq} with exponent $\alpha=0$. This proposition was in turn based on Theorem 2 of~\cite{Paden1987}. \cite{Cortes2006} later proposed a second-order Lyapunov criterion, which fundamentally boils down to $\mathcal{E}(t) = V(x(t))$ being strongly convex. Finally, Corollary 3.1 of~\cite{Hui2009} generalized the aforementioned Proposition 2.8 of~\cite{Cortes2005} to establish semistability. Indeed, these two results coincide for isolated equilibria.

We now present a novel result that generalizes the aforementioned first-order Lyapunov-based results, by exploiting our Lemma~\ref{lemma:Edotineq}. More precisely, given a Laypunov candidate function $V(\cdot)$, the objective is to set $\mathcal{E}(t) = V(x(t))$ and check that the conditions of Lemma~\ref{lemma:Edotineq} hold. To do this, and assuming $V$ to be locally Lipschitz continuous, we first borrow and adapt from~\cite{Bacciotti1999} the definition of \emph{set-valued time derivative} of $V:\mathcal{D}\to\mathbb{R}$ w.r.t. the differential inclusion~\eqref{eq:differentialinclusion}, given by
\begin{equation}
    \dot{V}(x) \triangleq \{a\in\mathbb{R}: \exists v\in K(x) \textnormal{ s.t. } a = p\cdot v,\, \forall p\in\partial V(x)\},
\end{equation}
for each $x\in\mathcal{D}$. Notice that, under Assumption~\ref{ass:K} For Filippov differential inclusions, we have $K = \mathcal{K}[F]$, and the set-valued time derivative of $V$ thus coincides with with the set-valued Lie derivative $\mathcal{L}_F V(\cdot)$. Indeed, more generally $\dot{V}$ could be seen as a set-valued Lie derivative $\dot{V} = \mathcal{L}_K V$ w.r.t. the set-valued map $K$.

\begin{definition}
\label{def:regularfunction}
We say that $V(\cdot)$ is said to be \emph{regular} if every directional derivative, given by
\begin{equation}
    V'(x;v) \triangleq \lim_{h\to 0}\frac{V(x+h\,v) - V(x)}{h},
\end{equation}
exists and is equal to
\begin{equation}
    V^\circ(x;v) \triangleq \limsup_{x'\to x\, h\to 0^+}\frac{V(x'+h\,v) - V(x')}{h},
\end{equation}
known as \emph{Clarke's upper generalized derivative}~\citep{Clarke1981}.
\end{definition}

\begin{assumption}
$V$ is locally Lipscthiz continuous and regular.
\label{ass:Vregularity}
\end{assumption}

In practice, regularity is a fairly mild and easy to guarantee condition. For instance, it would suffice that $V$ is convex or continuously differentiable to ensure that it is Lipschitz and regular. 

We are now equipped to formally establish the correspondence between the set-valued time-derivative of $V$ and the derivative of the energy function $\mathcal{E}(t) = V(x(t))$ associated with an arbitrary Carath\'{e}odory solution $x(\cdot)$ to the differential inclusion~\eqref{eq:differentialinclusion}.

\begin{lemma}[Lemma~1 of~\cite{Bacciotti1999}]
Under Assumption~\ref{ass:Vregularity}, given any Carath\'{e}odory solution $x:[0,\tau)\to\mathbb{R}^n$ to~\eqref{eq:differentialinclusion}, $\mathcal{E}(t)\triangleq V(x(t))$ is absolutely continuous and $\dot{\mathcal{E}}(t) = \frac{d}{dt}V(x(t))\in\dot{V}(x(t))$ a.e. in $t\in [0,\tau)$.
\label{lemma:dVdtisinVdot}
\end{lemma}

We are now ready to state and prove our Lyapunov-based sufficient condition for finite-time stability of differential inclusions.

\begin{theorem}
\label{thm:FTS}
Suppose that Assumptions~\ref{ass:K} and~\ref{ass:Vregularity} hold for some set-valued map $K:\mathbb{R}^n\rightrightarrows\mathbb{R}^n$ and some function $V:\mathcal{D}\to\mathbb{R}$, where $\mathcal{D}\subseteq\mathbb{R}^n$ is an open and positively invariant neighborhood of a point $x^\star\in\mathbb{R}^n$. Suppose that $V$ is positive definite w.r.t. $x^\star$ and that there exist constants $c>0$ and $\alpha < 1$ such that
\begin{equation}
\sup \dot{V}(x) \leq -c\,V(x)^\alpha
\label{eq:Vdotineq}
\end{equation}
a.e. in $x\in\mathcal{D}$. Then,~\eqref{eq:differentialinclusion} is finite-time stable at $x^\star$, with settling time upper bounded by
\begin{equation}
t^\star \leq \frac{V(x_0)^{1-\alpha}}{c(1-\alpha)},
\label{eq:settlingtimebound}
\end{equation}
where $x(0)=x_0$. In particular, any Carath\'{e}odory solution $x(\cdot)$ with $x(0)=x_0\in\mathcal{D}$ will converge in finite time to $x^\star$ under the upper bound~\eqref{eq:settlingtimebound}. Furthermore, if $\mathcal{D}=\mathbb{R}^n$, then~\eqref{eq:differentialinclusion} is globally finite-time stable. Finally, if $\dot{V}(x)$ is a singleton a.e. in $x\in\mathcal{D}$ and~\eqref{eq:Vdotineq} is exact, then so is~\eqref{eq:settlingtimebound}.
\end{theorem}

\begin{proof}
Note that, by Proposition~1 of~\cite{Bacciotti1999}, we know that~\eqref{eq:differentialinclusion} is Lyapunov stable at $x^\star$. All that remains to show is local convergence towards $x^\star$ (which must be an equilibrium) in finite time. Indeed, given any maximal solution $x:[0,t^\star)\to\mathbb{R}^n$ to~\eqref{eq:differentialinclusion} with $x(0)=x_0\neq x^\star$, we know by Lemma~\ref{lemma:dVdtisinVdot}, that $\mathcal{E}(t) = V(x(t))$ is absolutely continuous with $\dot{\mathcal{E}}(t)\in\dot{V}(x(t))$ a.e. in $t\in [0,t^\star]$. Therefore, we have
\begin{equation}
    \dot{\mathcal{E}}(t) \leq \sup\dot{V}(x(t)) \leq -c\,V(x(t))^\alpha = -c\,\mathcal{E}(t)^\alpha
\end{equation}
a.e. in $t\in [0,t^\star]$. Since $\mathcal{E}(0)=V(x_0) > 0$, given that $x_0\neq x^\star$, the result then follows by invoking Lemma~\ref{lemma:Edotineq} and noting that $\mathcal{E}(t^\star) = 0 \iff V(t^\star,x(t^\star)) = 0 \iff x(t^\star) = x^\star$.
\end{proof}

\section{Proof of the main result: Theorem 1}
\label{subsec:proofthm1}
Let us focus first on~\eqref{eq:GNF11}, since the proof for~\eqref{eq:GNF22} follows similar steps. The idea is to show that it is finite-time stable at $x^\star$, with the inequality in Theorem~\ref{thm:FTS} holding exactly for $V(x) = \|\nabla f(x)\|^2$. First, notice that $F(x)\triangleq -c\|\nabla f(x)\|^p\frac{[\nabla^2 f(x)]^r\nabla f(x)}{\nabla f(x)^\top[\nabla^2 f(x)]^{r+1}\nabla f(x)}$ is continuous near (but not at) $x=x^\star$, and undefined at $x=x^\star$ itself. Furthermore, we have 
\begin{subequations}
\begin{align}
    \|F(x)\| &= c\|\nabla f(x)\|^p\frac{\|[\nabla^2 f(x)]^r\nabla f(x)\|}{\nabla f(x)^\top[\nabla^2 f(x)]^{r+1}\nabla f(x)}\\
    &\leq c\|\nabla f(x)\|^p\frac{\lambda_\textnormal{max}(\nabla^2 f(x))^r\|\nabla f(x)\|}{\lambda_\textnormal{min}(\nabla^2 f(x))^{r+1}\|\nabla f(x)\|^2}\\
    &\leq c\frac{\lambda_\textnormal{max}(\nabla^2 f(x))^r}{\lambda_\textnormal{min}(\nabla^2 f(x))^{r+1}}\|\nabla f(x)\|^{p-1},
\end{align}
\end{subequations}
with $p-1\geq 0$ and $\lambda_\textnormal{min}[\nabla^2 f(x)] \geq m > 0$ everywhere near $x=x^\star$ for some $m>0$ ($m$-strong convexity). Therefore, $F$ is Lebesgue integrable (and thus measurable) and locally essentially bounded, which means that Assumption~\ref{ass:existenceFilippov} is satisfied.

Set $V(x) \triangleq \|\nabla f(x)\|^2$, defined over $\mathcal{D}$. If $\mathcal{D}$ is not positively invariant w.r.t.~\eqref{eq:GNF11}, we can always replace it by a smaller open subset that is, \emph{e.g.} a sufficiently small strict sublevel set contained within~$\mathcal{D}$. Clearly, $V$ is continuously differentiable, thus satisfying Assumption~\ref{ass:Vregularity}. Furthermore, it is positive definite w.r.t. $x^\star$ and, given $x\in\mathcal{D}\setminus\{x^\star\}$, we have
\begin{subequations}
\begin{align}
    \sup\dot{V}(x) &= \sup\{a\in\mathbb{R}: \exists v\in\mathcal{K}[F](x) \textnormal{ s.t. } a = p\cdot v, \forall p\in\partial V(x)\}\\
    &= \sup\left\{\left(2\nabla^2 f(x)\nabla f(x)\right)\cdot \left(-c\|\nabla f(x)\|^p \frac{[\nabla^2 f(x)]^r\nabla f(x)}{\nabla f(x)^\top[\nabla^2 f(x)]^{r+1}\nabla f(x)}\right) \right\}\\
    &= -2c\|\nabla f(x)\|^p\\
    &= -(2c)\,V(x)^\frac{p}{2}
\end{align}
\end{subequations}
with $\frac{p}{2}<1$. Furthermore, $\dot{V}(x^\star) = \{\nabla V(x^\star)\cdot v: v\in\mathcal{K}[F](x^\star)\} = \{0\}$ since $\nabla V(x^\star) = 2\nabla^2 f(x^\star)\nabla f(x^\star) = 0$. 
The result thus follows by invoking Theorem~\ref{thm:FTS}.

We now proceed to establish finite-time stability of~\eqref{eq:GNF22} at $x^\star$. Like before, we notice that $F(x) = -c\|\nabla f(x)\|_1^{p-1}x\frac{[\nabla^2 f(x)]^r\sign(\nabla f(x))}{\sign(\nabla f(x))^\top[\nabla^2 f(x)]^{r+1}\sign(\nabla f(x))}$ is continuous near, but not at, $x=x^\star$. Furthermore, notice that
\begin{subequations}
\begin{align}
    \|F(x)\| &= c\|\nabla f(x)\|_1^{p-1}\frac{\|[\nabla^2 f(x)]^r\sign(\nabla f(x))\|}{\sign(\nabla f(x))^\top[\nabla^2 f(x)]^{r+1}\sign(\nabla f(x))}\\
    &\leq c\|\nabla f(x)\|_1^{p-1}\frac{\lambda_\textnormal{max}(\nabla^2 f(x))^r\|\sign(\nabla f(x))\|}{\lambda_\textnormal{min}(\nabla^2 f(x))^{r+1}\|\sign(\nabla f(x))\|^2}\\
    &\leq c\frac{\lambda_\textnormal{max}(\nabla^2 f(x))^r}{\lambda_\textnormal{min}(\nabla^2 f(x))^{r+1}}\frac{\|\nabla f(x)\|_1^{p-1}}{\|\sign(\nabla f(x))\|}\\
    &\leq c\frac{\lambda_\textnormal{max}(\nabla^2 f(x))^r}{\lambda_\textnormal{min}(\nabla^2 f(x))^{r+1}}\frac{\|\nabla f(x)\|_1^{p-1}}{\|\sign(\nabla f(x))\|_1/\sqrt{n}}\\
    &\leq \frac{c}{\sqrt{n}}\frac{\lambda_\textnormal{max}(\nabla^2 f(x))^r}{\lambda_\textnormal{min}(\nabla^2 f(x))^{r+1}}\|\nabla f(x)\|_1^{p-1}\\
\end{align}
\end{subequations}
for every $x\in\mathcal{D}\setminus\mathcal{N}$, with $\mathcal{N}=\bigcup_{i=1}^n\left\{x\in\mathcal{D}:\frac{\partial f}{\partial x_i}(x) = 0\right\}$. Since $\mathcal{N}$ is a zero-measure set due to being a finite union of hypersurfaces in $\mathbb{R}^n$ ($n\geq 1$), and also recalling that $f$ is strongly convex near $x^\star$ and $p-1\geq 0$, it follows that $F$ is Lebesgue integrable and locally essentially bounded. Therefore, Assumption~\ref{ass:existenceFilippov} is once again satisfied.

Now consider the candidate Lyapunov function $V(x) = \|\nabla f(x)\|_1$, defined over $\mathcal{D}$. Clearly, $V$ is not continuously differentiable this time. However, it is still satisfies Assumption~\ref{ass:existenceFilippov} due to being a.e. differentiable. In particular, we have $\partial V(x) = \{\nabla^2 f(x)\sign(\nabla f(x))\}$ for every $x\in\mathcal{D}\setminus\mathcal{N}$. In other words, we have $\partial V(x) = \{\nabla^2 f(x)\sign(\nabla f(x))\}$ a.e. in $x\in\mathcal{D}$. 

Given $x\in\mathcal{D}\setminus\mathcal{N}$, we thus have
\begin{subequations}
\begin{align}
    \sup\dot{V}(x) &= \sup\{a\in\mathbb{R}:\exists v\in\mathcal{K}[F](x) \textnormal{ s.t. } a=p\cdot v, \forall p\in\partial V(x)\}\\
    &= \Big(\nabla^2 f(x)\sign(\nabla f(x))\Big)\cdot \left(-\frac{c\|\nabla f(x)\|_1^{p-1}[\nabla^2 f(x)]^r\sign(\nabla f(x))}{\sign(\nabla f(x))^\top[\nabla^2 f(x)]^{r+1}\sign(\nabla f(x))}\right)\\
    &= -c\|\nabla f(x)\|_1^{p-1}\\
    &= -c\, V(x)^{p-1},
\end{align}
\end{subequations}
with $p-1 < 1$. The result once again follows by invoking Theorem~\ref{thm:FTS}.





\end{document}